\documentclass[12pt]{article}

\usepackage{graphics}
\usepackage{epstopdf}
\usepackage{hyperref}
\hypersetup{
  colorlinks   = true,    
  urlcolor     = blue,    
  linkcolor    = blue,    
  citecolor    = red      
}
\usepackage[caption=false]{subfig}

\usepackage{inputenc}
\usepackage{amsmath}
\usepackage{csquotes}
\usepackage{tikz}
\usepackage{upgreek}
\usepackage{amssymb}
\usepackage{amsthm}
\usepackage{amsfonts}
\numberwithin{equation}{section}
\usepackage{booktabs,caption}
\usepackage[flushleft]{threeparttable}
\usepackage{stmaryrd}
\usepackage{wasysym}
\usepackage{makecell}
\usepackage{float}
\usepackage{colortbl} 
\usepackage{xcolor}

\newtheorem{theorem}{Theorem}[section]
\newtheorem{definition}[theorem]{Definition}
\newtheorem{proposition}[theorem]{Proposition}
\newtheorem{lemma}[theorem]{Lemma}

\newtheorem{corollary}[theorem]{Corollary}

\newtheorem{remark}[theorem]{Remark}

\theoremstyle{definition}

\usepackage{nccmath}
\usepackage{mathtools}
\usepackage[reftex]{theoremref}
\usepackage[left=1in, right=1in, top=1in, bottom=1in]{geometry}

\DeclareMathOperator{\der}{Der}
\DeclareMathOperator{\inn}{inn}
\DeclareMathOperator{\ch}{char}

\begin{document}


\title{Inner and Outer Derivations of $\mathbb{F}V_{8n}$}

\author{Praveen Manju and Rajendra Kumar Sharma}
\date{}
\maketitle

\begin{center}
\noindent{\small Department of Mathematics, \\Indian Institute of Technology Delhi, \\ Hauz Khas, New Delhi-110016, India$^{1}$}
\end{center}

\footnotetext[1]{{\em E-mail addresses:} \url{praveenmanjuiitd@gmail.com}(Corresponding Author: Praveen Manju), \url{rksharmaiitd@gmail.com}(Rajendra Kumar Sharma).}

\medskip

\begin{abstract}
Let $\mathbb{F}$ be a field of characteristic $0$ or an odd rational prime $p$. In this article, we give an explicit classification of all the inner and outer derivations of the group algebra $\mathbb{F}V_{8n}$, where $V_{8n}$ is a group of order $8n$ ($n$ a positive integer) with presentation $\langle a, b \mid a^{2n} = b^{4} = 1, ba = a^{-1}b^{-1}, b^{-1}a = a^{-1}b \rangle$. First, we explicitly classify all the $\mathbb{F}$-derivations of $\mathbb{F}V_{8n}$ by giving the dimension and a basis of the derivation algebra consisting of all $\mathbb{F}$-derivations of $\mathbb{F}V_{8n}$. Consequently, we classify all inner and outer derivations of $\mathbb{F}V_{8n}$ when $\mathbb{F}$ is an algebraic extension of a prime field. Thus, we establish that all the derivations of $\mathbb{F}V_{8n}$ are inner when the characteristic of $\mathbb{F}$ is $0$ or $p$ with $p$ relatively prime to $n$, and that non-zero outer derivations exist only in the case when the characteristic of $\mathbb{F}$ is $p$ with $p$ dividing $n$.
\end{abstract}



\section{Introduction}
Let $R$ be a ring. A derivation $d$ of $R$ is a map $d:R \rightarrow R$ satisfying $d(a+b) = d(a) + d(b)$ and $d(ab) = d(a)b + a d(b)$ for all $a, b \in R$. The map $d_{b}:R \rightarrow R$ ($b \in R$) defined by $d_{b}(a) = ab-ba$ for all $a \in R$, is a derivation of $R$. A derivation $d$ of $R$ is called inner if $d = d_{b}$ for some $b \in R$, and outer if it is not inner. If $S$ is a subring of $R$, then a derivation $d$ of $R$ is called an $S$-derivation if $d(s) = 0$ for all $s \in S$. We denote the set of all derivations of $R$ by $\der(R)$, the set of all inner derivations of $R$ by $\der_{\inn}(R)$, and the set of all $S$-derivations of $R$ by $\der_{S}(R)$. If $R$ has unity $1$, then $d(1) = 0$. If $R$ is commutative, then $\der(R)$ becomes an $R$-module concerning the componentwise sum and module action, and $\der_{\inn}(R)$ and $\der_{S}(R)$ become its submodules. If $G$ is a group, then the group ring $RG$ of $G$ over $R$ is defined as the set of all formal linear combinations of the form $\alpha = \sum_{g \in G} a_{g} g$ with finite support, where the support of $\alpha$ is defined to be the set $\text{supp}(\alpha) = \{g \in G \mid a_{g} \neq 0\}$. $RG$ is a ring concerning the componentwise addition and multiplication defined by $\alpha \beta = \sum_{g, h \in G} a_{g} b_{h} gh$ for $\alpha = \sum_{g \in G} a_{g} g$, $\beta = \sum_{g \in G} b_{g} g$ in $RG$. If $R$ is commutative with unity and $G$ is abelian, then $RG$ becomes a commutative unital algebra over $R$. If $R$ is commutative, then $\der(RG)$ is both an $RG$- as well as $R$-module concerning the componentwise sum and module action, and $\der_{\inn}(RG)$ and $\der_{S}(RG)$ are its submodules. Derivations play an essential role in many important branches of mathematics and physics. Understanding the structures of rings and algebras has involved extensive study of derivations. The notion of derivation in rings has been applied to various algebras, for example, BCI-algebras \cite{Muhiuddin2012}, von Neumann algebras \cite{Brear1992}, incline algebras \cite{Kim2014}, MV-algebras \cite{KamaliArdakani2013}, \cite{Mustafa2013}, Banach algebras \cite{Raza2016}, and lattices which play an important role in information theory and recovery, management of information access, and cryptanalysis \cite{Chaudhry2011}. Differentiable manifolds, operator algebras, $\mathbb{C}^{*}$-algebras, and representation theory of Lie groups are being studied using derivations \cite{Klimek2021}. Coding theory also makes use of derivations. For example, in \cite{Boucher2014}, codes have been built as modules over skew polynomial rings, where a derivation and an automorphism define multiplication. In \cite{Creedon2019}, authors have studied codes as images of derivations on group algebras. For more applications of derivations, we refer the reader to \cite{Atteya2019}.  

More research is expected on the derivations of group rings. In the literature, one can find only a few articles that concern derivations of group rings. The serious study of derivations of group rings begins with Smith's 1978 paper \cite{Smith1978}, in which the author studies derivations of a finitely generated torsion-free nilpotent group over a field. For instance, it is shown that such group rings always contain an outer derivation. The second major work in this direction can be seen in \cite{Spiegel1994} and \cite{MiguelFerrero1995}. In \cite{Spiegel1994}, it is shown that every derivation of an integral group ring is inner, and in \cite{MiguelFerrero1995}, it is shown that every $R$-derivation of the group ring $RG$ is inner if $R$ is a semiprime ring with $\ch(R)$, either $0$ or does not divide the order of any element of $G$, where $G$ is a torsion group with its center $Z(G)$ having finite index in it. Further evidence of the research in this direction is found only in 2019 in the works of Dishari Chaudhuri and Leo Creedon et al. Dishari Chaudhari, in her papers, \cite{Chaudhuri2019} and \cite{Chaudhuri2021}, generalizes the above results of \cite{Spiegel1994} and \cite{MiguelFerrero1995} respectively to $(\sigma, \tau)$-derivations of group rings of finite groups over a field and an integral domain where $\sigma, \tau$ satisfy certain conditions. In recent years, more tools have been used to study derivations of group rings. For instance, A. A. Arutyunov uses topological and character theory techniques to study derivations of group algebras (\cite{AleksandrAlekseev2020}, \cite{OrestD.Artemovych2020}, \cite{Arutyunov2021}, \cite{A.A.Arutyunov2020}, \cite{Arutyunov2023}, \cite{Arutyunov2020a}, \cite{Arutyunov2020}). In \cite{Creedon2019}, the authors study derivations of group rings in terms of the generators and relators of a group and, further, as applications, classify derivations of commutative group algebras over fields of positive characteristic and that of dihedral group algebras over a characteristic $2$ field. In \cite{Manju2023}, the authors explicitly classify all inner derivations of a group algebra $\mathbb{F}G$ of a finite group $G$ over an arbitrary field $\mathbb{F}$. This result generalizes Theorem $3.13$ of \cite{Creedon2019} where the authors have classified all inner derivations of dihedral group algebras over fields of characteristic $2$. In \cite{Manju2023}, the authors also classify all inner and outer derivations of dihedral, dicyclic, and semi-dihedral group algebras over fields of characteristic $0$ or an odd rational prime $p$. The work in this article concerns the derivations of the non-commutative algebra $\mathbb{F}V_{8n}$ of one of the essential non-abelian groups $V_{8n}$ over a field $\mathbb{F}$. The group ring $\mathbb{F}V_{8n}$ is important concerning its applications in coding theory (for reference, see \cite{Sehrawat2019}). Recently, the applications of derivations in coding theory have been discovered. For instance, in \cite{Creedon2019}, it has been found that some well-known codes occur as images of derivations. Thus, the derivations of the group ring $\mathbb{F}V_{8n}$ are interesting concerning their applications in coding theory; the explicit description of the derivation algebra of $\mathbb{F}V_{8n}$ can prove to be valuable in constructing some good group ring codes using different techniques.

We divide the article into three sections. Section \ref{section 2} includes some preliminary results and definitions that we need to prove our main results in Section \ref{section 3}. The results already proved in the literature are just stated (with reference) without proof, and the new results are stated with proof. In Section \ref{section 3}, we study the derivations of the group algebra $\mathbb{F}V_{8n}$, where $V_{8n} = \langle a, b \mid a^{2n} = b^{4} = 1, ba = a^{-1}b^{-1}, b^{-1}a = a^{-1}b \rangle$ is a group of order $8n$ ($n$ a positive integer) and $\mathbb{F}$, unless otherwise stated, is a field of characteristic $0$ or an odd rational prime $p$. We classify all the $\mathbb{F}$-derivations of $\mathbb{F}V_{8n}$ by giving the dimension and a basis of the derivation algebra $\der_{\mathbb{F}}(\mathbb{F}V_{8n})$ consisting of all $\mathbb{F}$-derivations of $\mathbb{F}V_{8n}$. Consequently, we classify all the derivations of $\mathbb{F}V_{8n}$ when $\mathbb{F}$ is an algebraic extension of a prime field. It is thus established that when $\mathbb{F}$ is an algebraic extension of a prime field, then all derivations of $\mathbb{F}V_{8n}$ are inner provided $\ch(\mathbb{F})$ is either $0$ or an odd prime $p$ with $\gcd(n,p)=1$ and non-zero outer derivations of $\mathbb{F}V_{8n}$ exist only when $\ch(\mathbb{F}) = p$ with $\gcd(n,p) \neq 1$. Here $\ch(\mathbb{F})$ denotes the characteristic of field $\mathbb{F}$ and $\gcd(n,p)$ denotes the greatest common divisor of $n$ and $p$.

\section{Preliminary Results}\label{section 2}
Let $R$ be a commutative unital ring and $G = \langle X \mid Y \rangle$ be a group with $X$ as its set of generators and $Y$ as the set of relators. In \cite{Creedon2019}, the authors have studied the derivations of the group ring $RG$ in terms of the generators and relators of the group $G$. They have given a necessary and sufficient condition under which a map $f:X \rightarrow RG$ can be extended to a derivation of $RG$:
\begin{theorem}[{\cite[Theorem 2.5]{Creedon2019}}]\th\label{theorem 2.1}
Let $G = \langle X \mid Y \rangle$ be a group with $X$ as its set of generators and $Y$ the set of relators. Let $F_{X}$ denote the free group on $X$ and $\phi: F_{X} \rightarrow G$ the natural onto homomorphism. Let $R$ be a commutative ring with unity and $f:X \rightarrow RG$ be a map. Then 
\begin{enumerate}
\item[(i)] $f$ can be uniquely extended to a map $\tilde{f}:F_{X} \rightarrow RG$ satisfying \begin{equation}\label{eq 2.1}
\tilde{f}(vw) = \tilde{f}(v) \phi(w) + \phi(v) \tilde{f}(w), \hspace{0.2cm} \forall \hspace{0.2cm} v, w \in F_{X}.
\end{equation}

\item[(ii)] $f$ can be uniquely extended to an $R$-derivation of $RG$ if and only if $\tilde{f}(y) = 0$ for all $y \in Y$.
\end{enumerate}
\end{theorem}
\noindent The map $\tilde{f}:F_{X} \rightarrow RG$ is defined as \begin{equation}\label{eq 2.2} \tilde{f}(x) = \begin{cases} 
f(x) & \hspace{0.2cm} \text{if $x \in X$} \\
-xf(x^{-1})x & \hspace{0.2cm} \text{if $x \in X^{-1}$} \\
0 & \hspace{0.2cm} \text{if $x = 1$}
\end{cases}\end{equation} and then on the whole of $F_{X}$ as \begin{equation}\label{eq 2.3} \tilde{f}(v) = \sum_{i=1}^{k} \left(\prod_{j=1}^{i-1} x_{j} \right) \tilde{f}(x_{i}) \left(\prod_{j=i+1}^{k} x_{j} \right),\end{equation} if $v = \prod_{i=1}^{k} x_{i}$. 
Any derivation $d$ of a group algebra $RG$ is uniquely determined by the image set $\{d(x) \mid x \in X\}$. If the generating set $X$ is finite, say, $X = \{x_{1}, ..., x_{t}\}$, then we denote such a derivation by $d_{(\bar{x_{1}}, ..., \bar{x_{t}})}$, where $\bar{x_{i}} = d(x_{i})$ ($1 \leq i \leq t$). The following theorem guarantees that for a $K$-algebra $\mathcal{A}$ where $K$ is an algebraic extension of a prime field $\mathbb{F}$, $\der(\mathcal{A}) = \der_{K}(\mathcal{A})$:
\begin{theorem}
[{\cite[Theorem 2.2]{Creedon2019}}]\th\label{theorem 2.2}
Let $\mathcal{A}$ be a $K$-algebra where $K$ is an algebraic extension of a prime field $\mathbb{F}$ and let $d \in \der(\mathcal{A})$. Then $d(K) = \{0\}$ and $d$ is a $K$-linear map.
\end{theorem}

In \cite{Manju2023}, the authors have given an explicit classification of all the inner derivations of a group algebra $\mathbb{F}G$ of a finite group $G$ over a field $\mathbb{F}$ of arbitrary characteristic: 
\begin{theorem}
[{\cite[Theorem 2.15]{Manju2023}}]\th\label{theorem 2.3}
Let $\mathbb{F}$ be a field of any characteristic and $G$ be a finite group of order $n$ having $r$ conjugacy classes. Then the dimension of $\der_{\inn}(\mathbb{F}G)$ over $\mathbb{F}$ is $n-r$. Moreover, if $C_{1}, ..., C_{r}$ are the all distinct conjugacy classes of $G$ with representatives $x_{1}, ..., x_{r}$ respectively and $s$ is a positive integer such that $|C_{i}| = 1$ for all $i \in \{1, ..., s\}$ and $|C_{i}| \geq 2$ for all $i \in \{s+1, ..., r\}$, then the set $$\mathcal{B}_{0} = \{d_{g} \mid g \in \bigcup_{i=s+1}^{r} (C_{i} \setminus \{x_{i}\})\}$$ forms a basis of $\der_{\inn}(\mathbb{F}G)$ over $\mathbb{F}$.
\end{theorem}

For a commutative unital ring $R$ and a group $G$, the homomorphism $\varepsilon:RG \rightarrow R$ defined by $\varepsilon(\sum_{g \in G}\lambda_{g}g) = \sum_{g \in G} \lambda_{g}$ is called the augmentation mapping of $RG$, and its kernel $\Delta(G) = \{\sum_{g \in G}\lambda_{g}g \in RG \mid \sum_{g \in G} \lambda_{g} = 0\}$ is called the augmentation ideal of $RG$. We require the following generalized notion: 
\begin{definition}\label{definition 2.4}
For a subgroup $H$ of $G$, define $$\Delta '(H) = \{\sum_{g \in G} \lambda_{g} g \in RG \mid \sum_{g \in H} \lambda_{g} = 0 \hspace{0.1cm} \text{\&} \hspace{0.1cm} \sum_{g \in G \setminus H} \lambda_{g} = 0\}.$$ Note that $\Delta '(H)$ is an $R$-submodule of the $R$-module $RG$ and $\Delta '(G) = \Delta(G)$. 
\end{definition}

For an element $\beta \in RG$, the subalgebra $C(\beta) = \{\alpha \in RG \mid \alpha \beta = \beta \alpha\}$ of $RG$ is called the centralizer of $\beta$ in $RG$. In this paper, we need another terminology:
\begin{definition}\label{definition 2.5}
We call the set $$\bar{C}(\beta) = \{\alpha \in \mathbb{F}G \mid \alpha \beta = - \beta \alpha\}$$  as the anti-centralizer of $\beta$ in $RG$, and it is an $R$-submodule of the $R$-module $RG$.
\end{definition} 

The proof of the following lemma can be found in \cite{Salahshour2020} (see Proposition $2.4$).
\begin{lemma}\th\label{lemma 2.6}
\begin{enumerate}
\item[(i)] If $n$ is even, then $V_{8n}$ has $2n+6$ conjugacy classes given by 
\begin{equation*}
\begin{aligned}
& \{1\}, ~~~ \{a^{n}\}, ~~~ \{b^{2}\}, ~~~ \{a^{n}b^{2}\}, ~~~ \{a^{2k},  a^{-2k}\} ~ for ~ 1 \leq k \leq \frac{n}{2} - 1, \\ &\quad \{a^{2k-1}, a^{-2k+1}b^{2}\} ~ for ~ 1 \leq k \leq n, ~~~ \{a^{2k}b^{2},  a^{-2k}b^{2}\} ~ \text{for} ~ 1 \leq k \leq \frac{n}{2}-1, \\ &\quad b^{V_{8n}} = \{a^{4k}b, a^{4k+2}b^{-1} \mid 1 \leq k \leq \frac{n}{2}\}, ~~~ (ab)^{V_{8n}} = \{a^{4k+1}b, a^{4k+3}b^{-1} \mid 1 \leq k \leq \frac{n}{2}\}, \\ &\quad (b^{3})^{V_{8n}} = \{a^{4k+2}b, a^{4k}b^{-1} \mid 1 \leq k \leq \frac{n}{2}\}, ~~~ (ab^{3})^{V_{8n}} = \{a^{4k+3}b, a^{4k+1}b^{-1} \mid 1 \leq k \leq \frac{n}{2}\}.
\end{aligned}
\end{equation*}

\item[(ii)] If $n$ is odd, then $V_{8n}$ has $2n+3$ conjugacy classes given by 
\begin{equation*}
\begin{aligned}
& \{1\}, ~~~ \{b^{2}\}, ~~~ \{a^{2k},  a^{-2k}\} ~ for ~ 1 \leq k \leq \frac{n-1}{2}, ~~~ \{a^{2k-1},  a^{-2k+1}b^{2}\} ~ for ~ 1 \leq k \leq n, \\ &\quad \{a^{2k}b^{2},  a^{-2k}b^{2}\} ~ for ~ 1 \leq k \leq \frac{n-1}{2}, ~~~ b^{V_{8n}} = \{a^{2k}b,  a^{2k}b^{-1} \mid 1 \leq k \leq n\}, \\ &\quad (ab)^{V_{8n}} = \{a^{2k+1}b,  a^{2k+1}b^{-1} \mid 1 \leq k \leq n\}.
\end{aligned}
\end{equation*}
\end{enumerate}
\end{lemma}

The following lemma is used in the proof of \th\ref{lemma 2.8}.
\begin{lemma}\th\label{lemma 2.7}
Let $\mathbb{F}$ be a field with $\ch(\mathbb{F}) \neq 2$, $G$ be a finite group, and $g \in G$ be an element of order $n$. If $n$ is even, then $\bar{C}(g^{-1}) = \bar{C}(g)$. Conversely, if $\bar{C}(g) \neq \{0\}$ and $\bar{C}(g^{-1}) = \bar{C}(g)$, then $n$ is even.
\end{lemma}
\begin{proof}
First, let $n$ be even and $\alpha \in \bar{C}(g)$. Then $\alpha g = -g \alpha$. By induction, $\alpha g^{k} = (-1)^{k} g^{k} \alpha$ for all $k \in \mathbb{N}$. So $\alpha g^{-1} = \alpha g^{n-1} = (-1)^{n-1}g^{n-1}\alpha = - g^{-1} \alpha$. Therefore, $\alpha \in \bar{C}(g^{-1})$. Conversely, let $\alpha \in \bar{C}(g^{-1})$. Now replacing the role of $g$ by that of $g^{-1}$ above (as $|g| = n = |g^{-1}|$), we get that $\alpha \in \bar{C}(g)$. 

For the converse of the lemma, let $\bar{C}(g) \neq \{0\}$ and $\bar{C}(g) = \bar{C}(g^{-1})$. Then some non-zero $\alpha \in \bar{C}(g)$ exists.  Since $\alpha g = - g \alpha$, as seen above, $\alpha g^{-1} = (-1)^{n-1} g^{-1} \alpha$. But again, since $\alpha \in \bar{C}(g^{-1})$, so $\alpha g^{-1} = -g^{-1} \alpha$. So we get that $(-1)^{n-1} g^{-1} \alpha = -g^{-1} \alpha$. Pre-multiplying both sides by $g$, we have $(-1)^{n-1}\alpha = - \alpha$. This gives $(-1)^{n-1} = -1$ because $\{\alpha\}$ is an $\mathbb{F}$-linearly independent subset of $\mathbb{F}G$. Therefore, $n$ is even.
\end{proof}

The following lemma is crucial in proving our main \th\ref{theorem 3.1}.
\begin{lemma}\th\label{lemma 2.8}
Let $\mathbb{F}$ be a field of characteristic $0$ or $p$, where $p$ is an odd rational prime. Then, the following statements hold:
\begin{enumerate}
\item[(i)] The set 
\begin{equation*}
\begin{aligned}
\bar{\mathcal{B}}(b) & = \{(a^{2k} - a^{-2k})b^{j} \mid 1 \leq k \leq \lfloor \frac{n-1}{2} \rfloor, 0 \leq j \leq 3\} \\ &\quad \cup \{(a^{2k-1} - a^{-(2k-1)}b^{2}), ~ (a^{2k-1} - a^{-(2k-1)}b^{2})b \mid 1 \leq k \leq \lfloor \frac{n+1}{2} \rfloor\} \\ &\quad \cup \{(a^{-(2k-1)} - a^{2k-1}b^{2}), ~ (a^{-(2k-1)} - a^{2k-1}b^{2})b \mid 1 \leq k \leq \lceil \frac{n-1}{2} \rceil\}
\end{aligned}
\end{equation*} is a basis of $\bar{C}(b)$ over $\mathbb{F}$.

\item[(ii)] $\bar{C}(b^{-1}) = \bar{C}(b)$.

\item[(iii)] The set 
\begin{equation*}
\begin{aligned}
\bar{\mathcal{B}}(a^{-1}b) & = \{(a^{2k} - a^{-2k}), ~ (a^{2k} - a^{-2k})b^{2}, ~ (a^{2k-1} - a^{-(2k+1)})b, ~ (a^{2k-1} - a^{-(2k+1)})b^{3} \\ &\quad \mid 1 \leq k \leq \lfloor \frac{n-1}{2} \rfloor\} \cup \{(a^{2k-1} - a^{-(2k-1)}b^{2}, ~ (a^{2k-2} - a^{-2k}b^{2})b \mid 1 \leq k \leq n\}
\end{aligned}
\end{equation*} is a basis of $\bar{C}(a^{-1}b)$ over $\mathbb{F}$.

\item[(iv)] $\bar{C}(ba) = \bar{C}(b^{-1}a)$.
\end{enumerate}
\end{lemma}

\begin{proof}
(i) It can be easily verified that $\bar{\mathcal{B}}(b)$ is an $\mathbb{F}$-linearly independent subset of $\mathbb{F}V_{8n}$. Now let $\alpha \in \bar{C}(b)$ and $\alpha = \sum_{i,j}\lambda_{i,j} a^{i}b^{j}$ for some $\lambda_{i,j} \in \mathbb{F}$ ($0 \leq i \leq 2n-1$, $0 \leq j \leq 3$). Put $\lambda_{(2n),j} = \lambda_{0,j}$ for $j \in \{0, 1, 2, 3\}$. Then from $\alpha b = -b \alpha$, we will get that $$\lambda_{i,0} = \begin{cases}
- \lambda_{(2n-i),2} & \text{if $i$ is odd} \\
- \lambda_{(2n-i),0} & \text{if $i$ is even}
\end{cases}, \hspace{1cm} \lambda_{i,1} = \begin{cases}
- \lambda_{(2n-i),3} & \text{if $i$ is odd} \\
- \lambda_{(2n-i),1} & \text{if $i$ is even}
\end{cases},$$ $$\lambda_{i,2} = \begin{cases}
- \lambda_{(2n-i),0} & \text{if $i$ is odd} \\
- \lambda_{(2n-i),2} & \text{if $i$ is even}
\end{cases}, \hspace{1cm} \lambda_{i,3} = \begin{cases}
- \lambda_{(2n-i),1} & \text{if $i$ is odd} \\
- \lambda_{(2n-i),3} & \text{if $i$ is even}
\end{cases}.$$
Using this information, we can calculate in a finite number of steps that when $n$ is even, 
\begin{equation*}
\begin{aligned}
\alpha & = \sum_{k=1}^{\frac{n}{2}} \lambda_{(2k-1),0} (a^{2k-1} - a^{-(2k-1)}b^{2}) + \sum_{k=1}^{\frac{n}{2}} \lambda_{(2n-(2k-1)),0} (a^{-(2k-1)} - a^{2k-1}b^{2}) \\ &\quad + \sum_{k=1}^{\frac{n}{2}} \lambda_{(2k-1),1}  (a^{2k-1} - a^{-(2k-1)}b^{2})b + \sum_{k=1}^{\frac{n}{2}} \lambda_{(2n-(2k-1)),1} (a^{-(2k-1)} - a^{2k-1}b^{2})b \\ &\quad + \sum_{k=1}^{\frac{n}{2}-1} \lambda_{(2k),0} (a^{2k} - a^{-2k}) + \sum_{k=1}^{\frac{n}{2}-1} \lambda_{(2k),1} (a^{2k} - a^{-2k})b \\ &\quad + \sum_{k=1}^{\frac{n}{2}-1} \lambda_{(2k),2} (a^{2k} - a^{-2k})b^{2} + \sum_{k=1}^{\frac{n}{2}-1} \lambda_{(2k),3} (a^{2k} - a^{-2k})b^{3}.
 \end{aligned}
\end{equation*} and when $n$ is odd,
\begin{equation*}
\begin{aligned}
\alpha & = \sum_{k=1}^{\frac{n+1}{2}} \lambda_{(2k-1),0} \left(a^{2k-1} - a^{-(2k-1)}b^{2}\right) + \sum_{k=1}^{\frac{n-1}{2}} \lambda_{(2n-(2k-1)),0} \left(a^{-(2k-1)} - a^{2k-1}b^{2}\right) \\ &\quad + \sum_{k=1}^{\frac{n+1}{2}} \lambda_{(2k-1),1} \left(a^{2k-1} - a^{-(2k-1)}b^{2}\right)b + \sum_{k=1}^{\frac{n-1}{2}} \lambda_{(2n-(2k-1)),1} \left(a^{-(2k-1)} - a^{2k-1}b^{2}\right)b \\ &\quad + \sum_{k=1}^{\frac{n-1}{2}} \lambda_{(2k),0} (a^{2k} - a^{-2k}) + \sum_{k=1}^{\frac{n-1}{2}} \lambda_{(2k),1} (a^{2k} - a^{-2k})b \\ &\quad + \sum_{k=1}^{\frac{n-1}{2}} \lambda_{(2k),2} (a^{2k} - a^{-2k})b^{2} + \sum_{k=1}^{\frac{n-1}{2}} \lambda_{(2k),3} (a^{2k} - a^{-2k})b^{3}.
\end{aligned}
\end{equation*}
So in both cases, $\alpha \in \bar{C}(b)$ can be written as an $\mathbb{F}$-linear combination of the elements in $\bar{\mathcal{B}}(b)$. Therefore, $\bar{\mathcal{B}}(b)$ is a basis of $\bar{C}(b)$ over $\mathbb{F}$.

(ii) The order of $b$ is $4$ which is even. Therefore, by \th\ref{lemma 2.7}, $\bar{C}(b^{-1}) = \bar{C}(b)$.

(iii) Let $\alpha \in \bar{C}(a^{-1}b)$ and $\alpha = \sum_{i,j} \lambda_{i,j} a^{i} b^{j}$ for some $\lambda_{i,j} \in \mathbb{F}$ ($0 \leq i \leq 2n-1$, $0 \leq j \leq 3$). From $\alpha (a^{-1}b) = -(a^{-1}b) \alpha$, we will get $$\sum_{i=0}^{2n-1}\lambda_{i,1} a^{i+1} = - \sum_{\substack{0 \leq i \leq 2n-1 \\ \text{$i$ odd}}} \lambda_{i,1} a^{-(i+1)} - \sum_{\substack{0 \leq i \leq 2n-1 \\ \text{$i$ even}}} \lambda_{i,3} a^{-(i+1)},$$ 
$$\sum_{i=0}^{2n-1}\lambda_{i,0} a^{i-1}b = - \sum_{\substack{0 \leq i \leq 2n-1 \\ \text{$i$ even}}} \lambda_{i,0} a^{-(i+1)}b - \sum_{\substack{0 \leq i \leq 2n-1 \\ \text{$i$ odd}}} \lambda_{i,2} a^{-(i+1)}b,$$ $$\sum_{i=0}^{2n-1} \lambda_{i,3} a^{i+1}b^{2} = - \sum_{\substack{0 \leq i \leq 2n-1 \\ \text{$i$ even}}} \lambda_{i,1} a^{-(i+1)}b^{2} - \sum_{\substack{0 \leq i \leq 2n-1 \\ \text{$i$ odd}}} \lambda_{i,3} a^{-(i+1)}b^{2},$$ $$\sum_{i=0}^{2n-1} \lambda_{i,2} a^{i-1}b^{3} = - \sum_{\substack{0 \leq i \leq 2n-1 \\ \text{$i$ odd}}} \lambda_{i,0} a^{-(i+1)}b^{3} - \sum_{\substack{0 \leq i \leq 2n-1 \\ \text{$i$ even}}} \lambda_{i,2} a^{-(i+1)}b^{3}.$$
From the first equation, we get the following: \begin{itemize}
\item[•] for each $i$ with $1 \leq i \leq 2n-1$ and $i$ odd, $\lambda_{(i-1),1} = - \lambda_{(2n-i-1),3}$;
\item[•] for each $i$ with $2 \leq i \leq 2n-2$, $i \neq n$ and $i$ even, $\lambda_{(i-1),1} = - \lambda_{(2n-i-1),1}$;
\item[•] $\lambda_{(n-1),1} = 0$ if $n$ is even and $\lambda_{(2n-1),1} = 0$.
\end{itemize} Similarly, from the third equation, we will get that \begin{itemize}
\item[•] for each $i$ with $1 \leq i \leq 2n-1$ and $i$ odd, $\lambda_{(i-1),3} = - \lambda_{(2n-i-1),1}$;
\item[•] for each $i$ with $2 \leq i \leq 2n-2$, $i \neq n$ and $i$ even, $\lambda_{(i-1),3} = - \lambda_{(2n-i-1),3}$;
\item[•] $\lambda_{(n-1),3} = 0$ if $n$ is even and $\lambda_{(2n-1),3} = 0$.
\end{itemize}
From the second equation, we get \begin{itemize}
\item[•] for each $i$ with $1 \leq i \leq 2n-1$ and $i$ odd, $\lambda_{i,0} = - \lambda_{(2n-i),2}$;
\item[•] for each $i$ with $2 \leq i \leq 2n-2$, $i \neq n$ and $i$ even, $\lambda_{i,0} = - \lambda_{(2n-i),0}$;
\item[•] $\lambda_{n,0} = 0$ if $n$ is even and $\lambda_{0,0} = 0$.
\end{itemize} Similarly, from the fourth equation, we will have \begin{itemize}
\item[•] for each $i$ with $1 \leq i \leq 2n-1$ and $i$ odd, $\lambda_{i,2} = - \lambda_{(2n-i),0}$;
\item[•] for each $i$ with $2 \leq i \leq 2n-2$, $i \neq n$ and $i$ even, $\lambda_{i,2} = - \lambda_{(2n-i),2}$;
\item[•] $\lambda_{n,2} = 0$ if $n$ is even and $\lambda_{0,2} = 0$.
\end{itemize} 
Now, using the above information, we can use a finite number of calculations to determine that \begin{equation*}
\begin{aligned}
\alpha & = \sum_{\substack{1 \leq i \leq 2n-1 \\ \text{$i$ odd}}} \lambda_{i,0} (a^{i} - a^{-i}b^{2}) + \sum_{\substack{2 \leq i \leq n-1 \\ \text{$i$ even}}} \lambda_{i,0} (a^{i} - a^{-i}) + \sum_{\substack{1 \leq i \leq 2n-1 \\ \text{$i$ odd}}}\lambda_{(i-1),1} (a^{i-1} - a^{2n-i-1}b^{2})b \\ &\quad + \sum_{\substack{2 \leq i \leq n-1 \\ \text{$i$ even}}}\lambda_{(i-1),1} (a^{i-1} - a^{2n-i-1})b + \sum_{\substack{2 \leq i \leq n-1 \\ \text{$i$ even}}}\lambda_{i,2} (a^{i} - a^{-i})b^{2} \\ &\quad + \sum_{\substack{2 \leq i \leq n-1 \\ \text{$i$ even}}}\lambda_{(i-1),3} (a^{i-1} - a^{2n-i-1})b^{3}.
              \end{aligned}
                 \end{equation*}
Therefore, the set $\bar{\mathcal{B}}(a^{-1}b)$ spans $\bar{C}(a^{-1}b)$ over $\mathbb{F}$. Also, $\bar{\mathcal{B}}(a^{-1}b)$ is an $\mathbb{F}$-linearly independent subset of $\bar{C}(a^{-1}b)$. Therefore, $\bar{\mathcal{B}}(a^{-1}b)$ is a basis of $\bar{C}(a^{-1}b)$ over $\mathbb{F}$.

(iv) Let $\alpha \in \bar{C}(ba)$. Then $\alpha (ba) = - (ba) \alpha$. So $\alpha (b^{-1}a) = b^{2}\alpha (ba) = b^{2} (-(ba)\alpha) = -(b^{-1}a)\alpha$ (as $b^{2} \in Z(V_{8n})$). Therefore, $\alpha \in \bar{C}(b^{-1}a)$. Conversely, let $\alpha \in \bar{C}(b^{-1}a)$. Then $\alpha (b^{-1}a) = - (b^{-1}a) \alpha$. So $\alpha (ba) =  b^{2}\alpha(b^{-1}a) = b^{2}(-(b^{-1}a)\alpha) = - (ba) \alpha$. Therefore, $\alpha \in \bar{C}(ba)$.
\end{proof}

\section{Main Results}\label{section 3}
The group $V_{8n}$ has the presentation: $$V_{8n} = \langle a, b \mid a^{2n} = b^{4} = 1, ba = a^{-1}b^{-1}, b^{-1}a = a^{-1}b \rangle.$$
So $V_{8n} = \{a^{i}b^{j} \mid 0 \leq i \leq 2n-1, 0 \leq j \leq 3\}$. In this section, we classify all the derivations, inner as well as outer, of the group algebra $\mathbb{F}V_{8n}$ over a field $\mathbb{F}$ of characteristic $0$ or an odd rational prime $p$.

\begin{theorem}\th\label{theorem 3.1}
Let $\mathbb{F}$ be a field and $p$ be an odd rational prime. 
\begin{enumerate}
\item[(i)] If $\ch(\mathbb{F}) = 0$ or $p$ with $\gcd(n,p)=1$, then the dimension of $\der_{\mathbb{F}}(\mathbb{F}V_{8n})$ over $\mathbb{F}$ is $3(2n-1)$ if $n$ is odd and $6(n-1)$ if $n$ is even, and a basis is $$\mathcal{B} = \{d_{(\bar{a}, \bar{b})} \mid (\bar{a}, \bar{b}) \in B_{1} \cup B_{2} \cup B_{3} \cup B_{4}\},$$ where 
\begin{equation*}
\begin{aligned}
B_{1} & = \{((a^{-2k} - a^{2k})b,0),   ~~   ((a^{2k-1} - a^{-(2k+1)})b,  (a^{2k} - a^{-2k})),  ~~   ((a^{-2k} - a^{2k})b^{3},0),  \\ &\quad  ((a^{2k-1} - a^{-(2k+1)})b^{3},  (a^{2k} - a^{-2k}) b^{2}), ~~  (0, (a^{2k} - a^{-2k}) b),   ~~  (0,(a^{2k} - a^{-2k}) b^{3})  \\ &\quad  \mid 1 \leq k \leq \lfloor \frac{n-1}{2} \rfloor\},
\\ B_{2} & = \{((a^{2k-2}-a^{-2k}b^{2})b, (a^{2k-1} - a^{-(2k-1)}b^{2})),  ~~  (0, (a^{-(2k-1)} - a^{2k-1}b^{2})b) \\ &\quad  \mid 1 \leq k \leq \lfloor \frac{n+1}{2} \rfloor\},
\end{aligned}
\end{equation*}
\begin{equation*}
\begin{aligned}
B_{3} & = \{((a^{-2k} - a^{2k-2}b^{2})b, (a^{-(2k-1)} - a^{2k-1}b^{2})),   ~~ (0,(a^{-(2k-1)} - a^{2k-1}b^{2})b) \\ &\quad \mid 1 \leq k \leq \lceil \frac{n-1}{2} \rceil\},
\\ B_{4} & = \{((a^{-(2k-1)} - a^{2k-1}b^{2})b,0) \mid 1 \leq k \leq n\}.
\end{aligned}
\end{equation*}

\item[(ii)] If $\ch(\mathbb{F}) = p$, with $\gcd(n,p) \neq 1$, then the dimension of $\der_{\mathbb{F}}(\mathbb{F}V_{8n})$ over $\mathbb{F}$ is $4(2n-1)$ if $n$ is odd and $8(n-1)$ if $n$ is even and a basis is $$\mathcal{B}' = \{d_{(\bar{a}, \bar{b})} \mid (\bar{a}, \bar{b}) \in  B_{1}' \cup B_{2}' \cup B_{3}' \cup B_{4}'\},$$ where
\begin{equation*}
\begin{aligned}
B_{1}' & = \{((a^{2k-1} - a^{-(2k+1)})b, (a^{2k} - a^{-2k})),       ~~ ((a^{2k+1} - a^{-(2k-1)}), (a^{2k} - a^{-2k}) b),   \\ &\quad     ((a^{2k-1} - a^{-(2k+1)})b^{3}, (a^{2k} - a^{-2k}) b^{2}),    ~~    ((a^{2k+1} - a^{-(2k-1)})b^{2}, (a^{2k} - a^{-2k}) b^{3}),    \\ &\quad ((a^{-2k} - a^{2k})b^{3},0),     ~~     ((a^{-2k} - a^{2k})b,0),     ~~     ((a^{-(2k-1)} - a^{2k+1})b^{2},0),    \\ &\quad     ((a^{-(2k-1)} - a^{2k+1}),0) \mid 1 \leq k \leq \lfloor \frac{n-1}{2} \rfloor\},
\\ B_{2}' & = \{((a^{2k-2} - a^{-2k}b)b^{2}, (a^{2k-1} - a^{-(2k-1)}b^{2})),   \\ &\quad    ((a^{2k} - a^{-(2k-2)}b^{2}),  (a^{2k-1} - a^{-(2k-1)}b^{2})b) \mid 1 \leq k \leq \lfloor \frac{n+1}{2} \rfloor\},
\\ B_{3}' & = \{((a^{-2k} - a^{2k-2}b^{2})b, (a^{-(2k-1)} - a^{2k-1}b^{2})),   \\ &\quad    ( (a^{-(2k-2)} - a^{2k}b^{2}), (a^{-(2k-1)} - a^{2k-1}b^{2})b)    \mid    1 \leq k \leq \lceil \frac{n-1}{2} \rceil\},    
\\ B_{4}' & = \{((a^{-(2k-1)} - a^{2k-1}b^{2})b,0),     ~~    ((a^{-(2k-2)} - a^{2k}b^{2}),0)  \mid 1 \leq k \leq n\}.
\end{aligned}
\end{equation*}
\end{enumerate}
\end{theorem}

\begin{proof}
The relators of the group $V_{8n}$ are $a^{2n}, ~ b^{4}, ~ (ba)^{2}, ~ (b^{-1}a)^{2}$. Let $f:X = \{a, b\} \rightarrow \mathbb{F}V_{8n}$ be a map that can be extended to an $\mathbb{F}$-derivation of $\mathbb{F}V_{8n}$. By \th\ref{theorem 2.1}, this is possible if and only if $\tilde{f}(a^{2n})=0, \hspace{0.1cm} \tilde{f}(b^{4}) = 0, \hspace{0.1cm} \tilde{f}((ba)^{2}) = 0, \hspace{0.1cm} \text{and} \hspace{0.1cm} \tilde{f}((b^{-1}a)^{2}) = 0$. Write $f(a)$ as $f(a) = c_{0} + c_{1}b + c_{2}b^{2} + c_{3}b^{3}$ for some $c_{i} \in \mathbb{F} \langle a \rangle$ ($0 \leq i \leq 3$). Then, using (\ref{eq 2.3}), we get
\begin{equation*}
\begin{aligned}
\tilde{f}(a^{2n}) & = 2na^{2n-1}(c_{0}+c_{2}b^{2}) + \sum_{\substack{1 \leq i \leq 2n \\ \text{$i$ odd}}} a^{2i-1} c_{1} b^{3} + \sum_{\substack{1 \leq i \leq 2n \\ \text{$i$ even}}} a^{2i-1} c_{3} b^{3} + \sum_{\substack{1 \leq i \leq 2n \\ \text{$i$ even}}} a^{2i-1} c_{1} b \\ &\quad + \sum_{\substack{1 \leq i \leq 2n \\ \text{$i$ odd}}} a^{2i-1} c_{3} b.
\end{aligned}
\end{equation*}
Since $c_{1}, c_{3} \in \mathbb{F}\langle a \rangle$, so $c_{1} = \sum_{i=0}^{2n-1} \lambda_{i} a^{i}$ and $c_{3} = \sum_{i=0}^{2n-1} \mu_{i} a^{i}$ for some $\lambda_{i}, \mu_{i} \in \mathbb{F}$ ($0 \leq i \leq 2n-1$). Put $\lambda_{2n} = \lambda_{0}$ and $\mu_{2n} = \mu_{0}$. Therefore, $\tilde{f}(a^{2n}) = 0$ if and only if the following three equations hold. \begin{equation}\label{eq 3.1} 2na^{2n-1}(c_{0}+c_{2}b^{2}) = 0,\end{equation}
\begin{equation}\label{eq 3.2} \sum_{\substack{1 \leq i \leq 2n \\ \text{$i$ odd}}} a^{2i-1} \left(\sum_{j=0}^{2n-1} \lambda_{j} a^{j}\right)  b^{3} + \sum_{\substack{1 \leq i \leq 2n \\ \text{$i$ even}}} a^{2i-1} \left(\sum_{j=0}^{2n-1} \mu_{j} a^{j}\right) b^{3} = 0,\end{equation}

\begin{equation}\label{eq 3.3} \sum_{\substack{1 \leq i \leq 2n \\ \text{$i$ even}}} a^{2i-1} \left(\sum_{j=0}^{2n-1} \lambda_{j} a^{j}\right) b + \sum_{\substack{1 \leq i \leq 2n \\ \text{$i$ odd}}} a^{2i-1} \left(\sum_{j=0}^{2n-1} \mu_{j} a^{j}\right) b = 0.\end{equation}
Now, there are two possibilities: (i) and (ii).

(i) $\ch(\mathbb{F}) = 0$ or $p$ with $\gcd(n,p)=1$: Then, from (\ref{eq 3.1}), $c_{0} = 0$ and $c_{2} = 0$. Therefore, $f(a) = c_{1}b + c_{3}b^{3}$. Post multiplying (\ref{eq 3.2}) by $ba$ and then comparing, we will get \begin{equation}\label{eq 3.4} \left(\sum_{\substack{0 \leq j \leq 2n-1 \\ \text{$j$ odd}}} \lambda_{j} a^{j}\right) \left(\sum_{\substack{1 \leq i \leq 2n \\ \text{$i$ odd}}} a^{2i}\right) + \left(\sum_{\substack{0 \leq j \leq 2n-1 \\ \text{$j$ odd}}} \mu_{j} a^{j}\right) \left(\sum_{\substack{1 \leq i \leq 2n \\ \text{$i$ even}}} a^{2i}\right) = 0\end{equation} and \begin{equation}\label{eq 3.5} \left(\sum_{\substack{0 \leq j \leq 2n-1 \\ \text{$j$ even}}} \lambda_{j} a^{j} \right)\left(\sum_{\substack{1 \leq i \leq 2n \\ \text{$i$ odd}}} a^{2i}\right) + \left(\sum_{\substack{0 \leq j \leq 2n-1 \\ \text{$j$ even}}} \mu_{j} a^{j} \right)\left(\sum_{\substack{1 \leq i \leq 2n \\ \text{$i$ even}}} a^{2i}\right) = 0.\end{equation}
We will get similar equations from (\ref{eq 3.3}) by just interchanging the roles of $\lambda_{i}$'s and $\mu_{i}$'s. Again, we have two possible cases: $n$ odd and $n$ even.

Case 1: $n$ is odd. Since $\sum_{\substack{1 \leq i \leq 2n \\ \text{$i$ odd}}} a^{2i} = \sum_{\substack{1 \leq i \leq 2n \\ \text{$i$ even}}} a^{2i} = \sum_{k=1}^{n} a^{2k}$ and $\sum_{k=1}^{n} a^{2k+j} = \sum_{\substack{1 \leq l \leq 2n \\ \text{$l$ odd}}} a^{l}$ for each odd $j$ in $\{0, 1, ..., 2n-1\}$, so by (\ref{eq 3.4}), $(\sum_{\substack{0 \leq j \leq 2n-1 \\ \text{$j$ odd}}} (\lambda_{j} + \mu_{j})) (\sum_{\substack{1 \leq l \leq 2n \\ \text{$l$ odd}}} a^{l}) = 0$. Therefore, \begin{equation}\label{eq 3.6} \sum_{\substack{0 \leq j \leq 2n-1 \\ \text{$j$ odd}}} (\lambda_{j} + \mu_{j}) = 0.\end{equation}
Further since $\sum_{k=1}^{n} a^{2k+j} = \sum_{\substack{1 \leq l \leq 2n \\ \text{$l$ even}}} a^{l}$ for each even $j$ in $\{0, 1, ..., 2n-1\}$, so by (\ref{eq 3.5}), $\left(\sum_{\substack{0 \leq j \leq 2n-1 \\ \text{$j$ even}}} (\lambda_{j} + \mu_{j}) \right) \left(\sum_{\substack{1 \leq l \leq 2n \\ \text{$l$ even}}} a^{l}\right) = 0$. Therefore,

\begin{equation}\label{eq 3.7} \sum_{\substack{0 \leq j \leq 2n-1 \\ \text{$j$ even}}} (\lambda_{j} + \mu_{j}) = 0.\end{equation}
(\ref{eq 3.6}) and (\ref{eq 3.7}) together imply that $c_{1} + c_{3} \in \Delta '(\langle a^{2} \rangle)$. The same outcome is produced from (\ref{eq 3.3}).

Case 2: $n$ is even. $\sum_{\substack{1 \leq i \leq 2n \\ \text{$i$ odd}}} a^{2i} = 2 \left(\sum_{\substack{1 \leq i \leq n \\ \text{$i$ odd}}} a^{2i}\right)$ and $\sum_{\substack{1 \leq i \leq 2n \\ \text{$i$ even}}} a^{2i} = 2 \left(\sum_{\substack{1 \leq i \leq n \\ \text{$i$ even}}} a^{2i}\right)$. Using this in (\ref{eq 3.4}), we can easily determine that

\begin{equation*}
\begin{aligned}
(\sum_{s=1}^{\frac{n}{2}} \lambda_{4s-3} + \sum_{t=1}^{\frac{n}{2}} \mu_{4t-1}) (\sum_{i=1}^{\frac{n}{2}} a^{4i-1}) + (\sum_{t=1}^{\frac{n}{2}} \lambda_{4t-1} + \sum_{s=1}^{\frac{n}{2}} \mu_{4s-3}) (\sum_{i=1}^{\frac{n}{2}} a^{4i-3}) = 0.
\end{aligned}
\end{equation*}

\noindent Therefore, \begin{equation}\label{eq 3.8}
\sum_{s=1}^{\frac{n}{2}} \lambda_{4s-3} + \sum_{t=1}^{\frac{n}{2}} \mu_{4t-1} = 0
\end{equation} and \begin{equation}\label{eq 3.9}
\sum_{t=1}^{\frac{n}{2}} \lambda_{4t-1} + \sum_{s=1}^{\frac{n}{2}} \mu_{4s-3} = 0
\end{equation}

Similarly, from (\ref{eq 3.5}), we will get 
\begin{equation*}
\begin{aligned}
(\sum_{s=1}^{\frac{n}{2}} \lambda_{4s-2} + \sum_{t=1}^{\frac{n}{2}} \mu_{4t}) (\sum_{i=1}^{\frac{n}{2}} a^{4i}) + (\sum_{s=1}^{\frac{n}{2}} \mu_{4s-2} + \sum_{t=1}^{\frac{n}{2}} \lambda_{4t}) (\sum_{i=1}^{\frac{n}{2}} a^{4i-2}) = 0. 
\end{aligned}
\end{equation*}

\noindent Therefore, \begin{equation}\label{eq 3.10}
\sum_{s=1}^{\frac{n}{2}} \lambda_{4s-2} + \sum_{t=1}^{\frac{n}{2}} \mu_{4t} = 0
\end{equation} and \begin{equation}\label{eq 3.11}
\sum_{t=1}^{\frac{n}{2}} \lambda_{4t} + \sum_{s=1}^{\frac{n}{2}} \mu_{4s-2} = 0
\end{equation}
We will get the same four equations when we work with (\ref{eq 3.3}). 

Now by (\ref{eq 2.1}), $$\tilde{f}(b^{4}) = \tilde{f}(b^{2})\phi(b^{2}) + \phi(b^{2})\tilde{f}(b^{2}) = \tilde{f}(b^{2})b^{2} + b^{2}\tilde{f}(b^{2}) =  2 b^{2} \tilde{f}(b^{2})$$ so that $\tilde{f}(b^{4}) = 0 \Leftrightarrow \tilde{f}(b^{2})=0$. 
Again, $$\tilde{f}(b^{2}) = \tilde{f}(b)\phi(b) + \phi(b) \tilde{f}(b) = f(b) b + b f(b)$$ so that $\tilde{f}(b^{2}) = 0 \Leftrightarrow f(b)b = -bf(b) \Leftrightarrow f(b) \in \bar{C}(b)$. Therefore, $\tilde{f}(b^{4}) = 0$ if and only if $f(b) \in \bar{C}(b)$.
Further, $$\tilde{f}((ba)^{2}) = \tilde{f}(ba) \phi(ba) + \phi(ba) \tilde{f}(ba) = \tilde{f}(ba) ba + ba \tilde{f}(ba)$$ so that $\tilde{f}((ba)^{2}) = 0 \Leftrightarrow \tilde{f}(ba) \in \bar{C}(ba)$. 
Similarly, $\tilde{f}((b^{-1}a)^{2}) = 0 \Leftrightarrow \tilde{f}(b^{-1}a) \in \bar{C}(b^{-1}a)$.
Now, $$\tilde{f}(ba) = \tilde{f}(b)\phi(a) + \phi(b)\tilde{f}(a) = f(b) a + b f(a)$$ so that $f(a) = b^{-1} \tilde{f}(ba) - b^{-1} f(b)a$. 
Also using (\ref{eq 2.2}) in addition to (\ref{eq 2.1}), we get $$\tilde{f}(b^{-1}a) = f(b^{-1}) \phi(a) + \phi(b^{-1}) f(a) = -b^{-1}f(b)b^{-1} a + b^{-1} f(a)$$ so that $f(a) = b \tilde{f}(b^{-1}a) + f(b)b^{-1}a$. Therefore, $$f(a) = b^{-1} \tilde{f}(ba) - b^{-1} f(b)a = b \tilde{f}(b^{-1}a) + f(b)b^{-1}a.$$

\begin{center}
$\Rightarrow b^{-1} \tilde{f}(ba) - b \tilde{f}(b^{-1}a) = \left(f(b)b^{-1} + b^{-1}f(b)\right) a$.
\end{center}

Now since $f(b) \in \bar{C}(b)$ and by \th\ref{lemma 2.8} (ii), $\bar{C}(b) = \bar{C}(b^{-1})$, therefore, $f(b) \in \bar{C}(b^{-1})$. This gives $f(b)b^{-1} + b^{-1}f(b) = 0$. Therefore, from above, we get that $\tilde{f}(b^{-1}a) = \tilde{f}(ba) b^{2}$. By \th\ref{lemma 2.8} (iv), $\bar{C}(ba) = \bar{C}(b^{-1}a)$.
Therefore, the conditions $\tilde{f}(ba) \in \bar{C}(ba)$ and $\tilde{f}(b^{-1}a) \in \bar{C}(b^{-1}a)$ are equivalent, provided the condition $f(b) \in \bar{C}(b)$ holds. 

Since $f(b) \in \bar{C}(b)$ and by \th\ref{lemma 2.8} (i), $\bar{\mathcal{B}}(b)$ is a basis of $\bar{C}(b)$, therefore, 

\begin{equation*}
\begin{aligned}
f(b) & = \sum_{k=1}^{\lfloor \frac{n-1}{2} \rfloor} \lambda_{k,0} (a^{2k} - a^{-2k}) + \sum_{k=1}^{\lfloor \frac{n-1}{2} \rfloor} \lambda_{k,1} (a^{2k} - a^{-2k}) b + \sum_{k=1}^{\lfloor \frac{n-1}{2} \rfloor} \lambda_{k,2} (a^{2k} - a^{-2k}) b^{2} \\ &\quad+ \sum_{k=1}^{\lfloor \frac{n-1}{2} \rfloor} \lambda_{k,3} (a^{2k} - a^{-2k}) b^{3} + \sum_{k=1}^{\lfloor \frac{n+1}{2} \rfloor} \mu_{k,0} (a^{2k-1} - a^{-(2k-1)}b^{2}) \\ &\quad + \sum_{k=1}^{\lfloor \frac{n+1}{2} \rfloor} \mu_{k,1} (a^{2k-1} - a^{-(2k-1)}b^{2})b + \sum_{k=1}^{\lceil \frac{n-1}{2} \rceil} \delta_{k,0} (a^{-(2k-1)} - a^{2k-1}b^{2}) \\ &\quad + \sum_{k=1}^{\lceil \frac{n-1}{2} \rceil} \delta_{k,1} (a^{-(2k-1)} - a^{2k-1}b^{2})b
\end{aligned}
\end{equation*} 
for some $\lambda_{k,j}$ ($1 \leq k \leq \lfloor \frac{n-1}{2} \rfloor$, $0 \leq j \leq 3$), $\mu_{k,0}, \mu_{k,1}$ ($1 \leq k \leq \lfloor \frac{n+1}{2} \rfloor$), $\delta_{k,0}, \delta_{k,1}$ ($1 \leq k \leq \lceil \frac{n-1}{2} \rceil$) in $\mathbb{F}$.
Further since $\tilde{f}(ba) \in \bar{C}(ba)$ and by \th\ref{lemma 2.8} (iii), $\bar{\mathcal{B}}(ba)$ is a basis of $\bar{C}(ba)$, therefore, \begin{equation*}
\begin{aligned}\tilde{f}(ba) & = \sum_{i=1}^{\lfloor \frac{n-1}{2} \rfloor} c_{k,0} (a^{2k} - a^{-2k}) + \sum_{i=1}^{\lfloor \frac{n-1}{2} \rfloor} c_{k,1} (a^{2k} - a^{-2k})b^{2} + \sum_{i=1}^{\lfloor \frac{n-1}{2} \rfloor} c_{k,2} (a^{2k-1} - a^{-(2k+1)})b \\ &\quad+ \sum_{i=1}^{\lfloor \frac{n-1}{2} \rfloor} c_{k,3} (a^{2k-1} - a^{-(2k+1)})b^{3} + \sum_{k=1}^{n} d_{k,0} (a^{2k-1} - a^{-(2k-1)}b^{2}) \\ &\quad + \sum_{k=1}^{n} d_{k,1} (a^{2k-2} - a^{-2k}b^{2})b\end{aligned}
\end{equation*} for some $c_{k,j}$ ($1 \leq k \leq \lfloor \frac{n-1}{2} \rfloor$, $0 \leq j \leq 3$), $d_{k,0}, d_{k,1}$ ($1 \leq k \leq n$) in $\mathbb{F}$.

As $f(a) = b^{-1}f(ba) - b^{-1}f(b)a$, $f(b) \in \bar{C}(b)$ and $b^{-1}a = a^{-1}b$, we get $f(a) = b^{-1}f(ba) + f(b)a^{-1}b$. So on comparing, $c_{1}b + c_{3}b^{3} = f(a) = b^{-1}f(ba) + f(b)a^{-1}b$, we will get the following four equations:
\begin{equation}\label{eq 3.12}
\begin{aligned}
& \sum_{k=1}^{\lfloor \frac{n-1}{2} \rfloor} c_{k,3} (a^{-(2k-1)} - a^{2k+1}) + \sum_{k=1}^{n} d_{k,1} a^{-(2k-2)} + \sum_{k=1}^{\lfloor \frac{n-1}{2} \rfloor} \lambda_{k,1} (a^{2k+1} - a^{-(2k-1)}) + \sum_{k=1}^{\lfloor \frac{n+1}{2} \rfloor} \mu_{k,1} a^{2k} \\ &\quad + \sum_{k=1}^{\lceil \frac{n-1}{2} \rceil} \delta_{k,1} a^{-(2k-2)} = 0;
\end{aligned}
\end{equation}

\begin{equation}\label{eq 3.13}
\begin{aligned}
c_{1} & = \sum_{k=1}^{\lfloor \frac{n-1}{2} \rfloor} c_{k,1} (a^{-2k} - a^{2k}) + \sum_{k=1}^{n} d_{k,0} a^{-(2k-1)} + \sum_{k=1}^{\lfloor \frac{n-1}{2} \rfloor} \lambda_{k,0} (a^{2k-1} - a^{-(2k+1)}) + \sum_{k=1}^{\lfloor \frac{n+1}{2} \rfloor} \mu_{k,0} a^{2k-2} \\ &\quad + \sum_{k=1}^{\lceil \frac{n-1}{2} \rceil} \delta_{k,0} a^{-2k};
\end{aligned}
\end{equation}

\begin{equation}\label{eq 3.14}
\begin{aligned}
& \sum_{k=1}^{\lfloor \frac{n-1}{2} \rfloor} c_{k,2} (a^{-(2k-1)} - a^{2k+1}) - \sum_{k=1}^{n} d_{k,1} a^{2k} + \sum_{k=1}^{\lfloor \frac{n-1}{2} \rfloor} \lambda_{k,3} (a^{2k+1} - a^{-(2k-1)}) - \sum_{k=1}^{\lfloor \frac{n+1}{2} \rfloor} \mu_{k,1} a^{-(2k-2)} \\ &\quad - \sum_{k=1}^{\lceil \frac{n-1}{2} \rceil} \delta_{k,1} a^{2k} = 0;
\end{aligned}
\end{equation}

\begin{equation}\label{eq 3.15}
\begin{aligned}
c_{3} & = \sum_{k=1}^{\lfloor \frac{n-1}{2} \rfloor} c_{k,0} (a^{-2k} - a^{2k}) -  \sum_{k=1}^{n} d_{k,0} a^{2k-1} + \sum_{k=1}^{\lfloor \frac{n-1}{2} \rfloor} \lambda_{k,2} (a^{2k-1} - a^{-(2k+1)}) - \sum_{k=1}^{\lfloor \frac{n+1}{2} \rfloor} \mu_{k,0} a^{-2k} \\ &\quad - \sum_{k=1}^{\lceil \frac{n-1}{2} \rceil} \delta_{k,0} a^{2k-2}.
\end{aligned}
\end{equation}

Now, we have two possible cases.

Case 1: $n$ is odd. Then $\lceil \frac{n-1}{2} \rceil = \frac{n-1}{2} = \lceil \frac{n-1}{2} \rceil$ and $\lfloor \frac{n+1}{2} \rfloor = \frac{n+1}{2}$. From (\ref{eq 3.12}), we get that 
\begin{equation*}
\begin{aligned}
& \sum_{k=1}^{\frac{n-1}{2}} (c_{k,3} - \lambda_{k,1})a^{-(2k-1)} + \sum_{k=1}^{\frac{n-1}{2}} (\lambda_{k,1} - c_{k,3})a^{2k+1} \\ & + \sum_{k=1}^{\frac{n-1}{2}} (d_{k,1} + \delta_{k,1}) a^{-(2k-2)} + \sum_{k=1}^{\frac{n+1}{2}} (d_{(n-k+1),1} + \mu_{k,1}) a^{2k} = 0.
\end{aligned}
\end{equation*} Therefore, $\lambda_{k,1} = c_{k,3}$ ($1 \leq k \leq \frac{n-1}{2}$), $\delta_{k,1} = - d_{k,1}$ ($1 \leq k \leq \frac{n-1}{2}$) and $\mu_{k,1} = - d_{(n-k+1),1}$ ($1 \leq k \leq \frac{n+1}{2}$).
From (\ref{eq 3.14}), we get an additional condition that $\lambda_{k,3} = c_{k,2}$ ($1 \leq k \leq \frac{n-1}{2}$). Also observe that for the above determined values of $c_{1}$ in (\ref{eq 3.13}) and $c_{3}$ in (\ref{eq 3.15}), we have that $c_{1} + c_{3} \in \Delta ' (\langle a^{2} \rangle)$. So
\begin{equation*}
\begin{aligned}
f(a) & = c_{1}b + c_{3}b^{3} = \sum_{k=1}^{\frac{n-1}{2}} c_{k,1} (a^{-2k} - a^{2k})b + \sum_{k=1}^{n} d_{k,0} (a^{-(2k-1)} - a^{2k-1}b^{2})b \\ &\quad + \sum_{k=1}^{\frac{n-1}{2}} \lambda_{k,0} (a^{2k-1} - a^{-(2k+1)})b + \sum_{k=1}^{\frac{n+1}{2}} \mu_{k,0} (a^{2k-2} - a^{-2k}b^{2})b + \sum_{k=1}^{\frac{n-1}{2}} \delta_{k,0} (a^{-2k} - a^{2k-2}b^{2})b \\ &\quad + \sum_{k=1}^{\frac{n-1}{2}} c_{k,0} (a^{-2k} - a^{2k})b^{3} + \sum_{i=1}^{\frac{n-1}{2}} \lambda_{k,2} (a^{2k-1} - a^{-(2k+1)})b^{3}.
\end{aligned}
\end{equation*} Now put \begin{equation*}
\begin{aligned}
\mathcal{B}_{o} & = \{d_{(\bar{a}, \bar{b})} \mid (\bar{a}, \bar{b}) \in B_{o1} \cup B_{o2} \cup B_{o3}\},
\end{aligned}
\end{equation*} where $\bar{a} = d(a) = f(a)$, $\bar{b} = d(b) = f(b)$ and \begin{equation*}
\begin{aligned}
B_{o1} & = \{((a^{-2k} - a^{2k})b,0),  ~~    ((a^{2k-1} - a^{-(2k+1)})b, (a^{2k} - a^{-2k})),   ~~   ((a^{-2k} - a^{2k})b^{3},0),   \\ &\quad  ((a^{2k-1} - a^{-(2k+1)})b^{3},  (a^{2k} - a^{-2k}) b^{2}),   ~~    (0, (a^{2k} - a^{-2k}) b),  ~~   (0,(a^{2k} - a^{-2k}) b^{3}),  \\ &\quad  ((a^{-2k} - a^{2k-2}b^{2})b, (a^{-(2k-1)} - a^{2k-1}b^{2})),  ~~  (0,(a^{-(2k-1)} - a^{2k-1}b^{2})b)    \mid 1 \leq k \leq \frac{n-1}{2}\},
\\ B_{o2} & = \{((a^{2k-2}-a^{-2k}b^{2})b, (a^{2k-1} - a^{-(2k-1)}b^{2})),  ~  (0, (a^{2k-1} - a^{-(2k-1)}b^{2})b) \\ &\quad \mid 1 \leq k \leq \frac{n+1}{2}\},
\\ B_{o3} & = \{((a^{-(2k-1)} - a^{2k-1}b^{2})b,0) \mid 1 \leq k \leq n\}.
\end{aligned}
\end{equation*}

Note that $|B_{o1}| = 4(n-1)$, $|B_{o2}| = n+1$ and $|B_{o3}| = n$. We now prove that $\mathcal{B}_{o}$ is a basis of $\der_{\mathbb{F}}(\mathbb{F}V_{8n})$ over $\mathbb{F}$. If $c_{1}, c_{2} \in \mathbb{F}$ and $\alpha_{1}, \alpha_{2}, \beta_{1}, \beta_{2} \in \mathbb{F}V_{8n}$, then 
\begin{eqnarray*}\left(c_{1}d_{(\alpha_{1}, \beta_{1})} + c_{2}d_{(\alpha_{2}, \beta_{2})}\right)(a) & = & c_{1}d_{(\alpha_{1}, \beta_{1})}(a) + c_{2}d_{(\alpha_{2}, \beta_{2})}(a) = c_{1}\alpha_{1} + c_{2}\alpha_{2} \\ & = & d_{(c_{1}\alpha_{1} + c_{2} \alpha_{2}, c_{1}\beta_{1} + c_{2} \beta_{2})}(a).\end{eqnarray*}
Similarly, $d_{(c_{1}\alpha_{1} + c_{2} \alpha_{2}, c_{1}\beta_{1} + c_{2} \beta_{2})}(b) = \left(c_{1}d_{(\alpha_{1}, \beta_{1})} + c_{2}d_{(\alpha_{2}, \beta_{2})}\right)(b)$. So by using the fact that any $\mathbb{F}$-derivation of $\mathbb{F}V_{8n}$ is completely determined by its image values $d(a)$ and $d(b)$, we get from above that $c_{1}d_{(\alpha_{1}, \beta_{1})} + c_{2}d_{(\alpha_{2}, \beta_{2})} = d_{(c_{1}\alpha_{1} + c_{2} \alpha_{2}, c_{1}\beta_{1} + c_{2} \beta_{2})}$. Therefore, $\mathcal{B}_{o}$ spans $\der_{\mathbb{F}}(\mathbb{F}V_{8n})$ over $\mathbb{F}$. Now let $r_{i}, s_{j}, t_{k} \in \mathbb{F}$ ($1 \leq i \leq 4(n-1)$, $1 \leq j \leq n+1$, $1 \leq k \leq n$) such that $$\sum_{(\alpha_{i}, \beta_{i}) \in B_{o1}} r_{i}d_{(\alpha_{i}, \beta_{i})} + \sum_{(\gamma_{j}, \delta_{j}) \in B_{o2}} s_{j}d_{(\gamma_{j}, \delta_{j})} +\sum_{(\zeta_{k}, \eta_{k}) \in B_{o3}} t_{k}d_{(\zeta_{k}, \eta_{k})} = 0,$$ where $0$ on the right denotes the zero derivation. Therefore,
\begin{equation*}
\begin{aligned}
& d_{\left(\sum_{(\alpha_{i}, \beta_{i}) \in B_{o1}} r_{i}\alpha_{i} + \sum_{(\gamma_{j}, \delta_{j}) \in B_{o2}} s_{j}\gamma_{j} + \sum_{(\zeta_{k}, \eta_{k}) \in B_{o3}} t_{k}\zeta_{k}, ~ \sum_{(\alpha_{i}, \beta_{i}) \in B_{o1}} r_{i}\beta_{i} + \sum_{(\gamma_{j}, \delta_{j}) \in B_{o2}} s_{j}\delta_{j} + \sum_{(\zeta_{k}, \eta_{k}) \in B_{o3}} t_{k}\eta_{k}\right)} \\ &\quad = 0.
\end{aligned}
\end{equation*}
\noindent This implies that the function on the left takes zero value at all points in the domain $\mathbb{F}V_{8n}$, in particular, at $a$ and $b$. Therefore,
$$\sum_{(\alpha_{i}, \beta_{i}) \in B_{o1}} r_{i}\alpha_{i} + \sum_{(\gamma_{j}, \delta_{j}) \in B_{o2}} s_{j}\gamma_{j} + \sum_{(\zeta_{k}, \eta_{k}) \in B_{o3}} t_{k}\zeta_{k} = 0$$ and $$\sum_{(\alpha_{i}, \beta_{i}) \in B_{o1}} r_{i}\beta_{i} + \sum_{(\gamma_{j}, \delta_{j}) \in B_{o2}} s_{j}\delta_{j} + \sum_{(\zeta_{k}, \eta_{k}) \in B_{o3}} t_{k}\eta_{k} = 0.$$
But then the $\mathbb{F}$-linear independence of the sets $\{\alpha_{i}, \gamma_{j}, \zeta_{k} \mid (\alpha_{i}, \beta_{i}) \in B_{o1}, (\gamma_{j}, \delta_{j}) \in B_{o2}, \\ (\zeta_{k}, \eta_{k}) \in B_{o3}\}$ and $\{\beta_{i}, \delta_{j}, \eta_{k} \mid (\alpha_{i}, \beta_{i}) \in B_{o1}, (\gamma_{j}, \delta_{j}) \in B_{o2}, (\zeta_{k}, \eta_{k}) \in B_{o3}\}$ implies that $r_{i} = 0$ ($1 \leq i \leq 4(n-1)$), $s_{j} = 0$ ($1 \leq j \leq n+1$) and $t_{k} = 0$ ($1 \leq k \leq n$). Therefore, $\mathcal{B}_{o}$ is an $\mathbb{F}$-linearly independent subset of $\der_{\mathbb{F}}(\mathbb{F}V_{8n})$. Therefore, $\mathcal{B}_{o}$ is a basis of $\der_{\mathbb{F}}(\mathbb{F}V_{8n})$ over $\mathbb{F}$ when $n$ is odd.
Further, since $|\mathcal{B}_{o}| = 3(2n-1)$, therefore, the dimension of $\der_{\mathbb{F}}(\mathbb{F}V_{8n})$ over $\mathbb{F}$ is $3(2n-1)$ when $n$ is odd.

Case 2: $n$ is even. Then $\lfloor \frac{n-1}{2} \rfloor = \frac{n}{2}-1$ and $\lceil \frac{n-1}{2} \rceil = \frac{n}{2}$. As in Case 1, we will get from (\ref{eq 3.12}) and (\ref{eq 3.14}) that $\lambda_{k,1} = c_{k,3}$ ($1 \leq k \leq \frac{n}{2}-1$), $\delta_{k,1} = -d_{k,1}$ ($1 \leq k \leq \frac{n}{2}$), $\mu_{k,1} = -d_{(n-k+1),1}$ ($1 \leq k \leq \frac{n}{2}$) and $\lambda_{k,3} = c_{k,2}$ ($1 \leq k \leq \frac{n}{2}-1$).
The relations (\ref{eq 3.8}), (\ref{eq 3.9}), (\ref{eq 3.10}), and (\ref{eq 3.11}) are also satisfied by the expressions of $c_{1}$ and $c_{3}$ determined in (\ref{eq 3.13}) and (\ref{eq 3.15}) respectively. So

\begin{equation*}
\begin{aligned}
f(a) & = c_{1}b + c_{3}b^{3} = \sum_{k=1}^{\frac{n}{2}-1} c_{k,1} (a^{-2k} - a^{2k})b + \sum_{k=1}^{n} d_{k,0} (a^{-(2k-1)} - a^{2k-1}b^{2})b \\ &\quad + \sum_{k=1}^{\frac{n}{2}-1} \lambda_{k,0} (a^{2k-1} - a^{-(2k+1)})b + \sum_{k=1}^{\frac{n}{2}} \mu_{k,0} (a^{2k-2}-a^{-2k}b^{2})b + \sum_{k=1}^{\frac{n}{2}} \delta_{k,0} (a^{-2k} - a^{2k-2}b^{2})b \\ &\quad + \sum_{k=1}^{\frac{n}{2}-1} c_{k,0} (a^{-2k} - a^{2k})b^{3} + \sum_{k=1}^{\frac{n}{2}-1} \lambda_{k,2} (a^{2k-1} - a^{-(2k+1)})b^{3}.
\end{aligned}
\end{equation*} Then, as shown in Case 1, it can be shown that the set \begin{equation*}
\begin{aligned}\mathcal{B}_{e} & = \{d_{(\bar{a}, \bar{b})} \mid (\bar{a}, \bar{b}) \in B_{e1} \cup B_{e2} \cup B_{e3}\}, \end{aligned}
\end{equation*} where 
\begin{equation*}
\begin{aligned}
B_{e1} & = \{((a^{-2k} - a^{2k})b,0),  ~~    ((a^{2k-1} - a^{-(2k+1)})b, (a^{2k} - a^{-2k})),   ~~   ((a^{-2k} - a^{2k})b^{3},0),   \\ &\quad  ((a^{2k-1} - a^{-(2k+1)})b^{3},  (a^{2k} - a^{-2k}) b^{2}),   ~~    (0, (a^{2k} - a^{-2k}) b),  ~~  (0,(a^{2k} - a^{-2k}) b^{3}) \\ &\quad \mid 1 \leq k \leq \frac{n}{2}-1\},
\\ B_{e2} & = \{((a^{2k-2}-a^{-2k}b^{2})b, (a^{2k-1} - a^{-(2k-1)}b^{2})),  ~~   ((a^{-2k} - a^{2k-2}b^{2})b, (a^{-(2k-1)} - a^{2k-1}b^{2})),  \\ &\quad   (0,(a^{2k-1} - a^{-(2k-1)}b^{2})b), ~~ (0, (a^{-(2k-1)} - a^{2k-1}b^{2})b)  \mid 1 \leq k \leq \frac{n}{2}\},  
\\ B_{e3} & = \{((a^{-(2k-1)} - a^{2k-1}b^{2})b,0) \mid 1 \leq k \leq n\},
\end{aligned}
\end{equation*} where $\bar{a} = d(a) = f(a)$ and $\bar{b} = d(b) = f(b)$, is a basis of $\der_{\mathbb{F}}(\mathbb{F}V_{8n})$ over $\mathbb{F}$. Since $|\mathcal{B}_{e}| = 6(n-1)$, therefore, the dimension of $\der_{\mathbb{F}}(\mathbb{F}V_{8n})$ over $\mathbb{F}$ is $6(n-1)$ when $n$ is even. 
Observe that $\mathcal{B} = \begin{cases}
\mathcal{B}_{o} & \text{if $n$ is odd} \\
\mathcal{B}_{e} & \text{if $n$ is even}
\end{cases}$. So, part (i) of the theorem is proved.\vspace{10pt}

(ii) $\ch(\mathbb{F}) \neq 0$, that is, $\ch(\mathbb{F}) = p$. Then the equation (\ref{eq 3.1}) always holds, that is, to say $c_{0}$ and $c_{2}$ are not necessarily $0$. We can proceed as in part  (i) to show that the conditions (\ref{eq 3.2}) and (\ref{eq 3.3}) give $c_{1} + c_{3} \in \Delta ' (\langle a^{2} \rangle)$ when $n$ is odd and give equations (\ref{eq 3.8}), (\ref{eq 3.9}), (\ref{eq 3.10}) and (\ref{eq 3.11}) when $n$ is even. Proceeding likewise as in (i), we will get that $c_{0} + c_{1}b + c_{2}b^{2} + c_{3}b^{3} = b^{-1}f(ba) + f(b)a^{-1}b$. On solving and equating, we will get the following four equations:

\begin{equation*}
\begin{aligned}
c_{0} & = \sum_{k=1}^{\lfloor \frac{n-1}{2} \rfloor} c_{k,3} (a^{-(2k-1)} - a^{2k+1}) + \sum_{k=1}^{n} d_{k,1} a^{-(2k-2)} + \sum_{k=1}^{\lfloor \frac{n-1}{2} \rfloor} \lambda_{k,1} (a^{2k+1} - a^{-(2k-1)}) \\ &\quad + \sum_{k=1}^{\lfloor \frac{n+1}{2} \rfloor} \mu_{k,1} a^{2k} + \sum_{k=1}^{\lceil \frac{n-1}{2} \rceil} \delta_{k,1} a^{-(2k-2)};
\\ c_{1} & = \sum_{k=1}^{\lfloor \frac{n-1}{2} \rfloor} c_{k,1} (a^{-2k} - a^{2k}) + \sum_{k=1}^{n} d_{k,0} a^{-(2k-1)} + \sum_{k=1}^{\lfloor \frac{n-1}{2} \rfloor} \lambda_{k,0} (a^{2k-1} - a^{-(2k+1)}) + \sum_{k=1}^{\lfloor \frac{n+1}{2} \rfloor} \mu_{k,0} a^{2k-2} \\ &\quad + \sum_{k=1}^{\lceil \frac{n-1}{2} \rceil} \delta_{k,0} a^{-2k};
\\ c_{2} & = \sum_{k=1}^{\lfloor \frac{n-1}{2} \rfloor} c_{k,2} (a^{-(2k-1)} - a^{2k+1}) - \sum_{k=1}^{n} d_{k,1} a^{2k} + \sum_{k=1}^{\lfloor \frac{n-1}{2} \rfloor} \lambda_{k,3} (a^{2k+1} - a^{-(2k-1)}) \\ &\quad - \sum_{k=1}^{\lfloor \frac{n+1}{2} \rfloor} \mu_{k,1} a^{-(2k-2)} - \sum_{k=1}^{\lceil \frac{n-1}{2} \rceil} \delta_{k,1} a^{2k};
\\ c_{3} & = \sum_{k=1}^{\lfloor \frac{n-1}{2} \rfloor} c_{k,0} (a^{-2k} - a^{2k}) -  \sum_{k=1}^{n} d_{k,0} a^{2k-1} + \sum_{k=1}^{\lfloor \frac{n-1}{2} \rfloor} \lambda_{k,2} (a^{2k-1} - a^{-(2k+1)}) - \sum_{k=1}^{\lfloor \frac{n+1}{2} \rfloor} \mu_{k,0} a^{-2k} \\ &\quad - \sum_{k=1}^{\lceil \frac{n-1}{2} \rceil} \delta_{k,0} a^{2k-2}.
\end{aligned}
\end{equation*}
For the above computed values of $c_{1}$ and $c_{3}$, a careful observation gives that when $n$ is odd, $c_{1} + c_{3} \in \Delta ' (\langle a^{2} \rangle)$ and when $n$ is even, then equations (\ref{eq 3.8}), (\ref{eq 3.9}), (\ref{eq 3.10}) and (\ref{eq 3.11}) hold. So $f(a) = c_{0} + c_{1}b + c_{2}b^{2} + c_{3}b^{3}$, where $c_{0}, c_{1}, c_{2}, c_{3}$ are as determined above. Then, as done in part (i), it can be very easily seen that the set $\mathcal{B}'$ given in part (ii) of the statement of the theorem is a basis of $\der_{\mathbb{F}}(\mathbb{F}V_{8n})$ over $\mathbb{F}$. Also, $|\mathcal{B}| = 8 (\lfloor \frac{n-1}{2} \rfloor) + 2(\lfloor \frac{n+1}{2} \rfloor)   +    2 (\lceil \frac{n-1}{2} \rceil)   +   2n$. Therefore, the dimension of $\der_{\mathbb{F}}(\mathbb{F}V_{8n})$ over $\mathbb{F}$ is $4(2n-1)$ when $n$ is odd and $8(n-1)$ when $n$ is even. Hence, the theorem is proved.
\end{proof}

\begin{corollary}\th\label{corollary 3.2}
Let $\mathbb{F}$ be an algebraic extension of a prime field such that $\ch(\mathbb{F})$ is either $0$ or an odd rational prime relatively prime to $n$. Then the dimension of the derivation algebra $\der(\mathbb{F}V_{8n})$ over $\mathbb{F}$ is $6(n-1)$ if $n$ is even and $3(2n-1)$ if $n$ is odd, and a basis is $\mathcal{B}$ given in \th\ref{theorem 3.1} (i).
\end{corollary}
\begin{proof}
Follows from \th\ref{theorem 2.2} and \th\ref{theorem 3.1} (i).
\end{proof}

\begin{proposition}\th\label{proposition 3.3}
Let $\mathbb{F}$ be a field. Then $\der_{\inn}(\mathbb{F}V_{8n})$ has dimension $6(n-1)$ if $n$ is even and $3(2n-1)$ if $n$ is odd.
\begin{enumerate}
\item[(i)] When $n$ is even, a basis for $\der_{\inn}(\mathbb{F}V_{8n})$ over $\mathbb{F}$ is 
\begin{equation*}
\begin{aligned}
\mathcal{B}_{1} & = \{d_{g} \mid g \in \{\{a^{2k}, ~ a^{2k}b^{2}, ~ a^{4k}b^{-1}, ~ a^{4k+1}b^{-1}, ~ a^{4k+2}b^{-1}, ~ a^{4k+3}b^{-1} \mid 1 \leq k \leq \frac{n}{2}-1\} \\ &\quad \cup \{a^{2k-1} \mid 1 \leq k \leq n\}  \cup \{a^{4k}b, ~ a^{4k+1}b, ~ a^{4k+2}b, ~ a^{4k+3}b \mid 1 \leq k \leq \frac{n}{2}\}\}.
\end{aligned}
\end{equation*}

\item[(ii)] When $n$ is odd, a basis for $\der_{\inn}(\mathbb{F}V_{8n})$ over $\mathbb{F}$ is 
\begin{equation*}
\begin{aligned}
\mathcal{B}_{2} & = \{d_{g} \mid g \in \{a^{2k}, ~ a^{2k}b^{2} \mid 1 \leq k \leq \frac{n-1}{2}\} \cup \{a^{2k-1}, ~ a^{2k}b, ~ a^{2k+1}b \mid 1 \leq k \leq n\} \\ &\quad \cup \{a^{2k}b^{-1}, ~ a^{2k+1}b^{-1} \mid 1 \leq k \leq n-1\}\}.
\end{aligned}
\end{equation*}
\end{enumerate}
\end{proposition}
\begin{proof}
Follows from \th\ref{theorem 2.3} and \th\ref{lemma 2.6}.
\end{proof}

\begin{corollary}\th\label{corollary 3.4}
Let $\mathbb{F}$ be an algebraic extension of a prime field and $p$ be an odd rational prime.
\begin{enumerate}
\item[(i)] If $\ch(\mathbb{F})$ is either $0$ or $p$ with $\gcd(n,p)=1$, then all derivations of $\mathbb{F}V_{8n}$ are inner, that is,  $\der(\mathbb{F}V_{8n}) = \der_{\inn}(\mathbb{F}V_{8n})$. In other words, $\mathbb{F}V_{8n}$ has no non-zero outer derivations.

\item[(ii)] If $\ch(\mathbb{F}) = p$ with $\gcd(n,p) \neq 1$, then $\der_{\inn}(\mathbb{F}V_{8n}) \subsetneq \der(\mathbb{F}V_{8n})$. In other words, $\mathbb{F}V_{8n}$ has non-zero outer derivations.
\end{enumerate}
\end{corollary}
\begin{proof}
Follows from \th\ref{theorem 3.1}, \th\ref{corollary 3.2} and \th\ref{proposition 3.3}.
\end{proof}\vspace{12pt}

We conclude this section with a remark on the restriction of derivations to the simple components in the Wedderburn decomposition of $\mathbb{F}V_{8n}$, in case it is semisimple.
\begin{remark}
Consider the case when the group algebra $\mathbb{F} V_{8n}$ is semisimple, that is, by Maschke's Theorem, the case when $\ch(\mathbb{F})$ does not divide the order of $V_{8n}$. Therefore, $\mathbb{F}V_{8n}$ is semisimple if and only if either $\ch(\mathbb{F}) = 0$ or $p$ with $\gcd(n,p) = 1$. By \th\ref{theorem 3.1}(i) and \th\ref{proposition 3.3}, $\der_{\mathbb{F}}(\mathbb{F}V_{8n}) = \der_{\inn}(\mathbb{F}V_{8n})$. As a result, in its Wedderburn decomposition, an $\mathbb{F}$-derivation of $\mathbb{F}V_{8n}$ restricts to each simple component of $\mathbb{F}V_{8n}$. More precisely, let $d \in \der_{\mathbb{F}}(\mathbb{F}V_{8n})$. Then $d = d_{\beta}$ for some $\beta \in \mathbb {F}V_{8n}$. A simple component of $\mathbb{F}V_{8n}$ looks like $\mathbb{F}V_{8n} e$, where $e$ is a primitive central idempotent of $\mathbb{F}V_{8n}$. Then, for any $\alpha \in \mathbb{F}V_{8n}e$, $\alpha = \gamma e$ for some $\gamma \in \mathbb{F}V_{8n}$. So $d(\alpha) = d_{\beta}(\alpha) = \alpha \beta - \beta \alpha = (\gamma e) \beta - \beta (\gamma e) = (\gamma \beta - \beta \gamma) e = (d_{\beta}(\gamma))e = d(\gamma) e \in \mathbb{F}V_{8n} e$ as $d(\gamma) \in \mathbb{F}V_{8n}$. Therefore, $d(\mathbb{F}V_{8n} e) \subseteq \mathbb{F}V_{8n} e$. When $\mathbb{F}$ is an algebraic extension of a prime field, then by \th\ref{corollary 3.2}, $\der_{\mathbb{F}}(\mathbb{F}V_{8n}) = \der(\mathbb{F}V_{8n})$, that is, every $\mathbb{F}$-derivation of $\mathbb{F}V_{8n}$ is a derivation of $\mathbb{F}V_{8n}$. Therefore, from the above discussion, it follows that when $\mathbb{F}$ is an algebraic extension of a prime field, then any derivation of $\mathbb{F}V_{8n}$ restricts to each simple component of $\mathbb{F}V_{8n}$, that is, $d(\mathbb{F}V_{8n} e) \subseteq \mathbb{F}V_{8n} e$ for all $d \in \der(\mathbb{F}V_{8n})$.

We can prove that the result holds for any semisimple group algebra $\mathbb{F}G$ of a finite group $G$ over a field $\mathbb{F}$. In this case, by Maschke's Theorem, $\ch(\mathbb{F})$ does not divide the order of $G$. By the Wedderburn-Artin Theorem, $$\mathbb{F}G \cong \mathbb{F}G e_{1} \oplus ... \oplus \mathbb{F}G e_{r}$$ with the following conditions:
\begin{enumerate}
\item[(i)] The above isomorphism is an isomorphism of $\mathbb{F}$-algebras.
\item[(ii)] $\{e_{1}, ..., e_{r}\}$ is a unique collection of primitive central idempotents of $\mathbb{F}G$.
\item[(iii)] For each $i \in \{1, ..., r\}$, $\mathbb{F}G e_{i}$ is a simple ring, and $\mathbb{F}Ge_{1}, ..., \mathbb{F}Ge_{r}$ are called simple components of $\mathbb{F}G$.
\item[(iv)] For each $i \in \{1, ..., r\}$, $\mathbb{F}G e_{i} \cong M_{n_{i}}(D_{i})$, the full matrix ring over $D_{i}$, where $D_{i}$ is a division algebra containing an isomorphic copy of $\mathbb{F}$ in its center.
\end{enumerate}

Let $d \in \der(\mathbb{F}G)$ and $i, j \in \{1, ..., r\}$ with $j \neq i$. Then $d(e_{i}) = d(e_{i}^{2}) = d(e_{i})e_{i} + e_{i}d(e_{i})$. Multiplying both sides by $e_{j}$ and using the fact that the idempotents $e_{i}, e_{j}$ are central with $e_{i}e_{j} = 0$, we get that $d(e_{i})e_{j} = d(e_{i})(e_{i}e_{j} + e_{i}e_{j}) = 0$. Therefore, for each $j \in \{1, ..., r\}$ with $j \neq i$, $d(e_{i})e_{j} = 0$, that is, for each $j \in \{1, ..., r\}$, $d(e_{i})$ is orthogonal to $e_{j}$. Therefore, $d(e_{i}) \in \mathbb{F}Ge_{i}$. Now let $\beta \in \mathbb{F}Ge_{i}$ so that $\beta = \alpha e_{i}$ for some $\alpha \in \mathbb{F}G$. Then $d(\beta) = d(\alpha e_{i}) = d(\alpha)e_{i} + \alpha d(e_{i})$. Since $\alpha, d(\alpha) \in \mathbb{F}G$, $e_{i}, d(e_{i}) \in \mathbb{F}Ge_{i}$ and $\mathbb{F}Ge_{i}$ is a two-sided ideal of $\mathbb{F}G$, therefore, $d(\beta) \in \mathbb{F}Ge_{i}$. Therefore, $d(\mathbb{F}Ge_{i}) \subseteq \mathbb{F}Ge_{i}$. Therefore, all derivations of $\mathbb{F}G$ restrict to all simple components of $\mathbb{F}G$ in the Wedderburn decomposition of $\mathbb{F}G$.
\end{remark}\vspace{12pt}

\noindent \textbf{Acknowledgements}\vspace{8pt}

\noindent The authors thank the referees and the editor for their critical reviews, comments, and suggestions, which have improved the presentation and quality of the paper. The first author receives a monthly institute fellowship from her institute, the Indian Institute of Technology Delhi, for carrying out her doctoral research. The second author is the ConsenSys Blockchain chair professor. He thanks ConsenSys AG for that privilege.
\bibliographystyle{plain}
\end{document}